\documentclass[12pt]{amsart}
\usepackage{amsmath, amssymb, amscd, amsfonts, amsxtra}
 \usepackage{tikz}
\usepackage[all,cmtip]{xy}
\usepackage{xparse,etoolbox}

\newcommand{\Mdef}[2]{\newcommand{#1}{\relax \ifmmode #2 \else $#2$\fi}}
\usepackage{hyperref}
  \hypersetup{colorlinks=true,citecolor=blue}
\CDat
\newtheorem{thm}{Theorem}[section]  
\newtheorem*{un-no-thm}{Theorem}
\newtheorem{cor}[thm]{Corollary}     
\newtheorem{lem}[thm]{Lemma}         
\newtheorem{prop}[thm]{Proposition}

\newtheorem{bigthm}{Theorem}
\newtheorem{bigcor}[bigthm]{Corollary}

\theoremstyle{definition}
\newtheorem{defn}[thm]{Definition}   

\theoremstyle{definition}

\theoremstyle{definition}
\theoremstyle{remark}
\newtheorem{rem}[thm]{Remark}

\newtheorem{notation}[thm]{Notation}
\newtheorem*{acks}{Acknowledgements}

\newtheorem*{cl}{Claim} 
\newtheorem*{out}{Outline}

\newtheorem*{intro-rem}{Remark}
\newtheorem*{intro-rems}{Remarks}

\newtheorem{ex}[thm]{Example}

\newcommand{\blocktheorem}[1]{%
  \csletcs{old#1}{#1}
  \csletcs{endold#1}{end#1}
  \RenewDocumentEnvironment{#1}{o}
    {\par\addvspace{1.5ex}
     \noindent\begin{minipage}{\textwidth}
     \IfNoValueTF{##1}
       {\csuse{old#1}}
       {\csuse{old#1}[##1]}}
    {\csuse{endold#1}
     \end{minipage}
     \par\addvspace{1.5ex}}
}

\blocktheorem{bigthm}

\newcommand{\R}{\mathbb{R}}
\newcommand{\Z}{\mathbb{Z}}
\newcommand{\Q}{\mathbb Q}

\newcommand{\ra}{\rightarrow}
\newcommand{\lra}{\longrightarrow}

\title[A higher Boltzmann distribution]{A higher Boltzmann distribution}
\author[M.~J. Catanzaro]{Michael J.\ Catanzaro}
\address{Dept.~ of Mathematics \\
          Wayne State University \\
	  Detroit, MI 48202}
\email{mike@math.wayne.edu}

\author[V.~Y. Chernyak]{Vladimir Y.\ Chernyak}
\address{Dept.~  of Chemistry\\
         Wayne State University\\
         Detroit, MI 48202}
\email{chernyak@chem.wayne.edu}

\author[J.~R. Klein]{John R.\ Klein}
\address{Dept.~  of Mathematics\\ Wayne State
	  University\\ Detroit, MI 48202}
\email{klein@math.wayne.edu}

\date{\today}
\begin{document}

\maketitle

\begin{abstract} We characterize the classical Boltzmann distribution
as  the unique solution to a certain combinatorial Hodge theory problem in homological degree zero on a finite graph. By substituting
for the graph a CW complex $X$ and a choice of degree $d \le \dim X$,
we define by direct analogy a higher dimensional 
Boltzmann distribution $\rho^B$ as
a certain real-valued cellular $(d-1)$-cycle.
We then give an explicit formula
for $\rho^B$. 

We explain how these ideas relate to the Higher Kirchhoff
Network Theorem of \cite{Catanzaro:Kirchhoff}. We also give an improved version of the Higher Matrix-Tree Theorems of \cite{Catanzaro:Kirchhoff}. 
\end{abstract}

\def\:{\colon}
\setcounter{tocdepth}{1}
\tableofcontents
\addcontentsline{file}{sec_unit}{entry}

\section{Introduction}
\subsection{Motivation} 
Physics and chemistry are rife with processes 
in which some quantity varies with time in 
a complicated, irregular way. The evolution of such a system is often
modeled by a so-called {\it master equation,} which takes the form
$\dot p = H p$. Here, $p$ is a time dependent
distribution on the states of the system and 
$H$, the master operator, governs 
 time evolution. 
The stationary distribution of a master equation 
is known as the {\it Boltzmann distribution} (or {\it Gibbs measure}).

For example, the classical Boltzmann distribution from thermodynamics governs how 
the system's states are populated, depending on their energy. 
Specifically, the probability of the system
 to be in state $j$ with energy $E_j$ is
proportional
to $e^{-\beta E_j}$, where $\beta = \tfrac{1}{k_BT}$, $T$ is the 
temperature and $k_B$ is the Boltzmann constant. The 
Boltzmann distribution
is then the normalized 
distribution
\[
  \rho^B = \frac{\sum_j e^{-\beta E_j} j}{\sum_j e^{-\beta E_j}}\, .
\]


 A 
Markov process with discrete state space
is conveniently described by a state diagram. The latter is
 a directed graph $\Gamma$ whose vertices label the states and whose edges
label the transition probabilities between states. Each directed edge
 $e$ of $\Gamma$ is equipped with a {\it transition rate}  
$k_{e}$.

A certain class of processes is obtained by the following procedure: 
start with a connected undirected
graph $X$ (i.e., a CW complex of dimension one) with vertex set $X_0$ and edge set $X_1$.
Then $\Gamma$ will be the {\it double} of $X$, i.e., the directed graph
having the same vertices, where each edge of $X$ is now replaced
by a pair of opposing directed edges.  Each directed edge of $\Gamma$ is given
by a pair $(i,\alpha)$, where
$i$ is a vertex of $X$ and $\alpha$ is an edge of $X$ that is incident to $i$.
The transition rates of the process are given as follows: choose
a real number $\beta >0$ (inverse temperature)
and functions
 $E\: X_0 \to \Bbb R$ and $W\: X_1 \to \Bbb R$.
Then the transition rate across  $(i,\alpha)$ 
is given by 
\[
k_{i,\alpha} := e^{\beta(E_i - W_{\alpha})}\, .
\]
When the transition rates are written in this way, the process is said be in {\it Arrhenius form}.

The process just described is not completely general since
$\Gamma$ is a double. In fact, the above is an
example of a process that is in
{\it detailed balance} (or time reversible)
 in the sense that there exists a distribution
$\pi\: \Gamma_0 \to \Bbb R_+$ such that 
\begin{equation} \label{eqn:detailed-balance}
\pi(i) k_{i,\alpha} = \pi(j)k_{j,\alpha}\, ,
\end{equation}
for any edge $\alpha$ of $X$ with joining a pair of vertices $i$ and $j$ 
(see e.g., \cite[ch.~1]{Kelly}; in our example, take $\pi(i) = e^{-\beta E_i}$).
The number $\pi(i) k_{i,\alpha}$ is the probability flux across
$(i,\alpha)$. Equation \eqref{eqn:detailed-balance}
says that the pair of states $i$ and $j$
are in equilibrium along $\alpha$.
It is not difficult to show that every process (with discrete time and finite state space) in detailed balance can be written in Arrhenius form.

If $X$ is finite,
the Boltzmann distribution 
arises as the normalized vector
\[
\tfrac{1}{Z}\sum_{i\in X_0} \pi(i)i\, , \qquad Z := \sum_{i\in X_0} \pi(i)\, .
\]
Consequently, the Boltzmann distribution
is a normalized equilibrium distribution for the process.

The higher Boltzmann distribution developed in this paper is associated with
a Markov process on a CW complex which is a natural generalization of the process
we described above on a graph. With the intention of providing
motivation, the process is described below. However,
we will not investigate its properties here as is not within
the current scope.

Let $X$ be a finite connected CW complex,  let $d\le  \dim X$ be a fixed positive integer, and let $X_k$ denote the set of $k$-cells.
With respect to a mild technical assumption, we will show how associate a Markov process. 
A state of the process
will be an integer-valued cellular $(d-1)$-cycle $\zeta$ on $X$ in a fixed homology class.
Roughly, a directed edge from  $\zeta$ to another state $\zeta'$ is given by an 
``elementary homology'' in the sense that 
\[
\zeta' = \zeta + u \partial e\, ,
\]
 where $u$ is an integer and $e$ is a $d$-cell that is incident to $\zeta$ at 
a specified $(d-1)$-cell $f$. We also require that $e$ not be incident to $\zeta'$ at $f$.

When $d = \dim X = 1$, it will turn out that the process coincides with 
the one constructed above. 
Before giving the details, we explain the technical assumption on $X$,
which amounts a  homological condition
that will eventually be removed as it is not  required for the results of this paper:

\begin{defn} The CW complex $X$ is  {\it $d$-pseudo-regular} if
for every $d$ cell $e \in X_d$ and any $(d-1)$-cell $f \in X_{d-1}$, we have
\[
\langle \partial e , f \rangle \in \{0,\pm 1\}\, ,
\]
where the above denotes the coefficient of $f$ in
the boundary of $e$.
If $X$ is regular then it is $d$-pseudo-regular for all $d$.
In particular any connected polytopal complex
\cite[chap.~5]{Ziegler} or any connected finite simplicial complex
is $d$-pseudo-regular.
If $d=\dim X =1$, then $X$ is automatically $d$-pseudo-regular.
\end{defn} 

Assuming $X$ is $d$-pseudo-regular, we can now describe the process.
Consider the oriented graph in which a vertex is given by
an integer $(d-1)$-cycle $\zeta\in Z_{d-1}(X;\Z)$. 
A  directed edge from a vertex $\zeta$ to a vertex $\zeta'$ 
is given by a pair
\[
(f,e)
\]
in which
\begin{itemize}
\item $e \in X_d$ is a $d$-cell and $f\in X_{d-1}$ is a $(d-1)$-cell;
\item $\langle \zeta,f\rangle \ne 0 \ne \langle \partial e,f\rangle \ $;
\item $\langle \zeta',f\rangle = 0$.
\end{itemize}
Note that the above conditions imply  
\[
\zeta' \,\, =\,\,  \zeta - 
\langle \zeta,f\rangle \langle \partial e,f\rangle \partial e \, ,\]
so $\zeta'$ is obtained from $\zeta$ by means of an elementary homology.
We denote the directed edge by $(f,e,\zeta)$.
 
We now fix a base vertex 
\[
\hat x \in Z_{d-1}(X;\Z)\, ,
\]
such that the homology class $[\hat x] \in H_{d-1}(X;\Z)$ is non-trivial.
Consider the subgraph 
generated by $\hat x$ in the sense that a vertex $\zeta$
lies in this subgraph if there exists a finite sequence 
of directed edges $\hat x = \zeta_0 \to \zeta_1 \to \cdots \to \zeta_{k} =\zeta$, i.e., a finite directed
path from $\hat x$ to $\zeta$. 
An  edge belongs to this subgraph  if it occurs in such a path. 
We denote this subgraph by
 \[
 \Gamma_{X,\hat x}\, .
 \]

\begin{defn} The {\it cycle incidence graph} of  $(X,\hat x)$ is
the directed graph $\Gamma_{X,\hat x}$.
\end{defn}

\begin{ex} Suppose $X$ has dimension one and 
 $\hat x := i$ is any vertex of $X$. 
In this situation
 $\Gamma_{X,i}$ coincides with the double of $X$.
 \end{ex}

\begin{rem} If $\dim X > 1$ then the cycle-incidence graph
is usually not a double, so the process 
won't be detailed balance. 
Furthermore, $\Gamma_{X,\hat x}$ typically  has infinitely many vertices.
\end{rem} 

\begin{ex}\label{ex:the-example} Let $X$ be the 
two dimensional torus with
regular cell structure indicated by the following picture:
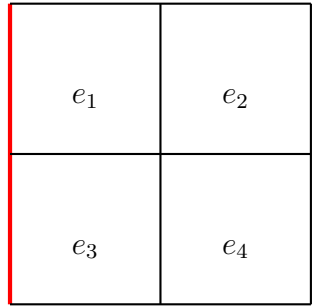
\begin{center}
\begin{tikzpicture}[scale=2]
\draw ;
\draw [red] [ultra thick] (0,0) -- (0,1) -- (0,2);
\draw [thick] (1,0) -- (1,1) -- (1,2);
\draw [thick] (2,0) -- (2,1) -- (2,2);
\draw [thick] (0,0) -- (1,0) -- (2,0);
\draw [thick] (0,1) -- (1,1) -- (2,1);
\draw [thick] (0,2) -- (1,2) -- (2,2);
\node[align=left, below] at (.5,1.5)%
{$e_1$};
\node[align=left, below] at (1.5,1.5)%
{$e_2$};
\node[align=left, below] at (.5,.5)%
{$e_3$};
\node[align=left, below] at (1.5,.5)%
{$e_4$};
\end{tikzpicture}
\end{center}
Then $X$ is obtained from the picture  by identifying
the opposite sides of the outer square.
Consequently, there are $4$ two-cells ($e_1,\dots, e_4$), 8 one-cells and 4 zero-cells. Our initial one-cycle $\zeta_0 = \hat x$ is the meridian given by the red segment. In this case the cycle incidence graph has infinitely many vertices. To see this 
start at $\hat x$ and jump across $e_1$. Modulo a choice of
orientations for the cells, the jump results in the cycle
given by $\zeta_1 = \hat x+ \partial e_1$, and the latter is incident to 
the $e_2$. Jumping across $e_2$ results in the cycle $\zeta_2 = \hat x + \partial e_1 - \partial e_2$ which is incident to $e_1$. Moreover
$\zeta_2$ is distinct from $\hat x$. Iterating this procedure
by jumping first across $e_1$ and then across $e_2$, results
in an infinite number of distinct vertices $\zeta_0,\zeta_1,\zeta_2,\dots$ (with $\zeta_{2k+1} = \hat x + (2k+1)\partial e_1 - (2k)\partial e_2$
and $\zeta_{2k} = \hat x + (2k)(\partial e_1 - \partial e_2)$)
 in the cycle incidence graph of $(X,\hat x)$.
\end{ex}

To complete the description of the process, we must label
the edges of $\Gamma_{X,\hat x}$ with transition rates. This is done as follows:
choose  functions
$E\: X_{d-1} \to \R$ and $W\: X_d \to \R$ and an inverse temperature
$\beta > 0$.
Then the transition rate of the directed edge $(f,e,\zeta)$ is given by
\begin{equation} \label{eqn:transition-rate}
k_{f,e,\zeta} :=  e^{\beta(W_f - E_e)}\, .
\end{equation}
The description of the process is now complete.

Summarizing, as the initial $(d-1)$-cycle evolves with respect to the process, 
it jumps from cycle to cycle by adding/subtracting the boundaries of $d$-cells, analogous
to a particle `jumping' over an edge on a graph. That is, the evolving cycle does not change its homology class, and so
each homology class will have its own distinct dynamics. 

One can interpret
the higher Boltzmann distribution defined in this paper 
as the long-time limit of the {\it average
distribution} within each homology class. By ``average,'' we mean the
average over the stochastic 
process in the sense of probability theory. 
This will be pinned-down in a subsequent paper.

\subsection{The combinatorial Hodge problem} We now turn to the problem
of formulating the higher Boltzmann distribution in terms of combinatorial Hodge theory.
Let $X$ be a finite CW complex and let $d \le \dim X$ be as above. 
Henceforth, we do {\it not}
require $X$ to be $d$-pseudo-regular.
Let $C_*(X;\R)$ be the cellular chain complex of $X$
with real coefficients. Each $C_j(X;\R)$
 is equipped with a standard inner product $\langle{-},{-} \rangle$ which is determined
by declaring $X_j$ to be an orthonormal basis.

\begin{defn} \label{eqn:Riemann-structure}
A {\it scalar structure} on  $X$ consists  of functions 
\[
E_j: X_{j} \to \mathbb{R},\qquad j = 0,1,\dots
\]
When $j = d$ we typically write $E := E_{d-1}$ and $W := E_d$.
\end{defn}
Let $\beta > 0$ be given.
Then the function $E_j$ 
equips $C_j(X;\R)$ with a {\it modified inner product}
\[
  \langle x, y \rangle_{E_j} := e^{\beta E_j(x)} \langle x, y \rangle
   \qquad x, y \in X_{j}\, .
\]
In our next formulation, we only make use of the function $E = E_{d-1}$.
Define the formal adjoint
\[
\partial_E^*\: C_{d-1}(X;\R)\to C_d(X;\R) 
\] 
using the standard inner product on $C_d(X;\R)$
and the modified inner product on $C_{d-1}(X;\R)$,
i.e.,  $\partial_E^* = \partial^*e^{\beta E}$, where
$\partial^\ast$  is the formal adjoint in the standard inner product structures.

\begin{defn} Let $X$ and $E$ be as above.
The {\em combinatorial Hodge problem} in degree $d-1$ is 
  the following: given $x \in H_{d-1}(X;\R)$, find an explicit 
  formula for the unique cycle $\rho \in Z_{d-1}(X;\R)$ such that
  \begin{itemize}
    \item $\rho$ represents $x$, and 
    \item $\rho$ is co-closed, i.e., $\partial_E^* \rho = 0$.
  \end{itemize}
\end{defn}

The condition that $\rho$ be co-closed can be re-stated as
the assertion that $\rho$ should
be orthogonal to any boundary with respect to the modified inner product:
\[
  \langle \alpha, \partial_E^* \rho \rangle = \langle \partial \alpha, \rho \rangle_E = 0\, ,
\]
for any $\alpha \in X_d$.

\begin{rem}\label{rem:splitting}
The combinatorial Hodge problem for $(X,E)$ is equivalent to finding
 an orthogonal splitting 
of the quotient homomorphism $p$ appearing in the short exact sequence
\begin{equation*}
 0 \lra B_{d-1}(X;\R) \lra Z_{d-1}(X;\R) \stackrel{p}{\lra} H_{d-1}(X;\R) \lra 0,
\end{equation*}
with respect to the modified inner product on $Z_{d-1}(X;\R) \subset C_{d-1}(X;\R)$.
\end{rem}

The original Hodge problem asks to find a unique harmonic
representative for any cohomology class on a compact,
orientable Riemannian manifold. By relaxing the hypotheses
to a connected CW complex, we are able to write down an
explicit formula.  Our solution to the combinatorial Hodge problem 
will involve a summation over {\it spanning co-trees}. 
The latter are certain subcomplexes of $X$ of dimension $d-1$ which
are a higher dimensional analog of the vertices of a graph. They
are homologically dual to the higher dimensional 
spanning trees of \cite{Catanzaro:Kirchhoff} and hence their name.
In the following definition, $\beta_k(X) = \dim H_k(X;\mathbb{Q})$ denotes the 
$k$th Betti number of $X$, and $X^{(k)}$ denotes the $k$-skeleton of $X$.

\begin{defn} \label{defn:spanning} A $(d-1)$-dimensional
{\em spanning co-tree} for $X$ is a 
  subcomplex $L \subset X$ such that
  \begin{enumerate}
  \item[\rm (a)] The inclusion $i_L\: L \subset X$ induces an isomorphism
  \[
  i_{L\ast}\: H_{d-1}(L;\Q) @> \cong >> H_{d-1}(X;\Q)\, ;\]
  \item[\rm (b)] $\beta_{d-2}(L) = \beta_{d-2}(X)$;
  \item[\rm (c)] $X^{(d-2)} \subset L \subset X^{(d-1)}$.
  \end{enumerate}
Clearly, the number of spanning co-trees is finite.
\end{defn}

\begin{rem} 
 Equivalently, conditions (a)-(c) are equivalent to conditions (a'),(b) and (c),
where
\begin{enumerate}
\item[\rm (a')] The relative homology group $H_{d-1}(X,L;\Q)$ is trivial.
\end{enumerate}
\end{rem}

\begin{rem} 
We are indebted to a referee for pointing out that
our notion of spanning co-tree is equivalent to that of a
{\it relatively acyclic complex} in \cite{DKM2} as well as
to the {\it complement of a cobase} in 
\cite[p.~156]{Lyons}.
\end{rem} 

Spanning co-trees come packaged with auxiliary data that will be used for obtaining
the desired splitting. Observe that the projection $Z_{d-1}(L;\Z) \to H_{d-1}(L;\Z)$ 
is an isomorphism since $L$ has no $d$-cells. Let
$\phi_L$ be the composite
\begin{equation}
  \phi_L\:Z_{d-1}(L;\Z) @> \cong >> H_{d-1}(L;\mathbb Z) @> i_{L\ast} >> H_{d-1}(X;\mathbb Z)\, .
\end{equation}
Then $\phi_L$ becomes an isomorphism after tensoring
with the rational numbers by the defining properties of $L$. 
We invert $\phi_L$ rationally to obtain a homomorphism of rational vector spaces
\begin{equation}
  \psi_L\: H_{d-1}(X;\Q) @> (\phi_L \otimes \Q)^{-1}>> Z_{d-1}(L;\Q)
  @> i_{L\ast}>> Z_{d-1}(X;\Q)\, .
\end{equation}
Since  $H_{d-1}(X,L;\Q)$ is trivial, the group $H_{d-1}(X,L;\Bbb Z)$ is finite.
Let
\begin{equation} \label{eqn:aL}
a_L := |H_{d-1}(X,L;\Bbb Z)|
\end{equation}
denote its order.

We define the {\em weight} of $L$ to be the real number
\begin{equation} \label{eqn:tauL}
  \tau_L = \tau_L(E) := a_L^2 \prod_{b \in L_{d-1}} e^{-\beta E_{b}}\, .
\end{equation}

With respect to the above definitions, we can state
the solution to the combinatorial Hodge problem:

\begin{bigthm}[Boltzmann Splitting Formula] 
\label{thm:Hodgeprob} Let $(X,E)$  and $x\in H_{d-1}(X;\R)$ be as above.
  Then the solution to the combinatorial Hodge problem is given by
  $\rho =\Psi(x)$, in which $\Psi\: H_{d-1}(X;\R)\to Z_{d-1}(X;\R)$ is the homomorphism
  \[
    \tfrac{1}{\nabla} \sum_L \tau_L \psi_L\, , 
  \]
  where the sum runs over all $(d-1)$-dimensional spanning co-trees $L$ and $\nabla = \sum_L \tau_L$. 
 \end{bigthm}

Theorem \ref{thm:Hodgeprob} enables us to define the higher Boltzmann distribution:

\begin{defn} \label{defn:Hodgeprob} Let $(X,E)$ be as above 
and let  $x \in H_{d-1}(X;\Z)$ be an integer homology class.
The {\it higher Boltzmann distribution}
at $x$
is the real $(d-1)$-cycle
\[
\rho^B := \tfrac{1}{\nabla}\sum_L \tau_L \psi_L(\bar x) \in Z_{d-1}(X;\R)\, ,
\]
where $\bar x \in H_{d-1}(X;\Q)$ is the image of $x$ with respect to the homomorphism
$H_{d-1}(X;\Z) \to H_{d-1}(X;\Q)$.
\end{defn}

\begin{rem}\label{rem:classicalBoltz}
  The classical Boltzmann distribution  is a
  special case of Definition ~\ref{defn:Hodgeprob}: 
let $X$ be finite connected graph 
 and take $d=1$.  Then the spanning co-trees of $X$ are given by 
  the vertices. We take $x\in H_0(X;\Z)\cong \Z$ to be the canonical generator
  (given by choosing a vertex of $X$; the generator is independent of this choice).
  For a vertex $L = j$  the normalized weight is given by 
 \[
 \nabla^{-1}\tau_L \,\, = \,\,   Z^{-1}e^{-\beta E_j}\, ,
 \qquad Z = \sum_{i\in X_0} e^{-\beta E_i} \, ,
 \] 
since $\phi_L$ is an integral isomorphism. Then $\psi_L(\bar x) = j$ and 
the assertion follows.
\end{rem}

\begin{rem} \label{rem:E=0} When $E = 0$,
the coefficients $\tau_L$ are rational numbers and the map 
$\Psi$ is defined as a homomorphism of rational vector spaces
$H_{d-1}(X;\Q)\to Z_{d-1}(X;\Q)$. If we further assume that
$x$ is a rational homology class, then
the solution to the combinatorial Hodge problem
gives an explicit expression for the 
Harmonic ``forms'' with respect to the combinatorial Laplacian
$-(\partial\partial^\ast+\partial^\ast \partial)\:  C_{d-1}(X;\Q) \to C_{d-1}(X;\Q)$.
\end{rem}


\begin{rem}
The proof we give of Theorem \ref{thm:Hodgeprob}
is an application of the theory of generalized inverses
to the quotient map $p \:Z_{d-1}(X;\R) \ra H_{d-1}(X;\R)$
(cf.\ \cite{Moore:Reciprocal}, \cite{Penrose:general}, \cite{Ben-Israel:GenInv}). In the late 1980s a summation formula was given for the Moore-Penrose pseudo-inverse
 \cite{Berg:Three}, \cite{Ben-Tal:GeomProp}; this is the formula we make use of.
When writing the current paper, we also came to realize that 
by applying the summation formula to split the inclusion map
\[
Z_{d}(X;\R) @> \subset >> C_d(X;\R)\, ,
\]
one gets another proof of our  higher dimensional analog of Kirchhoff's theorem on electrical networks \cite{Catanzaro:Kirchhoff} (see Remark 
\ref{rem:oldKthm} below).
\end{rem}

\begin{rem} The main application
of Theorem~\ref{thm:Hodgeprob} will appear
in the first author's
Ph.~D. thesis \cite{Catanzaro:thesis}. 
\end{rem}

\begin{rem} It is tempting to speculate whether
a result like Theorem \ref{thm:Hodgeprob} 
holds in the case of a Riemannian manifold. For this, one 
needs notion of spanning co-tree adapted to the space of 
differential forms. Then the sum appearing in the Theorem~\ref{thm:Hodgeprob} 
would presumably be replaced by a convergent infinite series of operators
indexed over the set of spanning co-trees. This would give
an explicit solution to the main result of classical Hodge theory.
\end{rem}

\subsection{Asymptotic behavior}  
Given $(X,E)$ and $d$ as above, with $E$ suitably generic, it turns out that
sum of Theorem \ref{thm:Hodgeprob} is asymptotic as a function of
 $\beta$ to the term with highest weight.
To explain this,
let $\mathcal F^\ast := \mathcal F^\ast_{d-1}(X)$ denote the  set of 
$(d-1)$-dimensional spanning co-trees of $X$.
Let
\[
E_{\mathcal F^\ast}\:  \mathcal F^\ast \to \mathbb R
\]
be the functional given by
\[
L \mapsto \sum_{b\in L_{d-1}} E_b\, .
\]
A scalar function $E\: X_{d-1} \to \R$ is said to be
{\it non-degenerate} if it is one-to-one. This is clearly a generic
condition.
  
As spanning co-trees are a matroid basis 
\cite[p.~156]{Lyons},  a greedy algorithm shows that
if $E$ is non-degenerate then the function $E_{\mathcal F^\ast}$
possesses a unique minimum $L^{\mu}$.\footnote{See e.g.,
\cite[\S1.8]{Oxley}.  A previous draft of the paper had stronger assumptions on $E$. We are
grateful to a referee for pointing out to us that non-degeneracy suffices.} 

As the parameter $\beta$ tends to $\infty$,
it is easy to check that the 
 operator $\Psi$ appearing in  Theorem~\ref{thm:Hodgeprob} is asymptotic to
 $\psi_{L^{\mu}}$ when we consider these as vector-valued functions with components
 indexed over $\mathcal F^\ast$ 
(see the proof of \cite[lem.~3.6]{CKS}). More precisely, we have

\begin{bigcor} Let  $(X,E)$ and $x\in H_{d-1}(X;\Z)$ be as above. Assume in addition that
 $E\: X_{d-1} \to \R$ is  non-degenerate. 
 Then we have
\[
\lim_{\beta \to \infty} \frac{\tau_L}{\nabla} =
\begin{cases}
0 \mbox{\phantom{asdf}} L \ne L^\mu ;\\
1 \mbox{\phantom{asdf}} L = L^\mu\, .
\end{cases}
\]
In particular, $\rho^B(x)$ is asymptotic in $\beta$ to 
the rational $(d-1)$-cycle $\psi_{L^\mu}(x)$. 
Consequently, in the low
temperature limit\footnote{The limit $\beta\to \infty$  is called the {\it low temperature limit} (cf.\ \cite{Catanzaro:Kirchhoff}).}
 the higher Boltzmann distribution rationally quantizes.
\end{bigcor}

\subsection{Improved higher matrix-tree theorems}
Let $X$ be a finite connected CW complex 
equipped with scalar structure $E_\ast$. Fix $d\le \dim X$.
As above, we set $E = E_{d-1}$ and $W = E_d$.

\begin{defn} Let $\partial_{E,W}^\ast \:C_{d-1}(X;\R) \to C_{d}(X;\R)$ be the linear transformation given by
\[
e^{-\beta W} \partial^\ast e^{\beta E}\, ,
\]
where \begin{itemize}
\item $\partial^\ast\: C_{d-1}(X;\R) \to C_{d}(X;\R)$ is the formal adjoint
to the boundary operator with respect to the standard inner product;
\item  $e^{-\beta W}\: C_{d}(X;\R) \to C_{d}(X;\R)$ is given by $b \mapsto e^{-\beta W_b}$ for $b\in X_d$;
\item  $e^{\beta E}\: C_{d-1}(X;\R) \to C_{d-1}(X;\R)$ is given by $f\mapsto e^{\beta E_f}$ for $f\in X_{d-1}$.
\end{itemize}
The {\it (restricted) biased Laplacian} is the operator
\[
\mathcal L^{E,W} := \partial \partial_{E,W}^\ast\: B_{d-1}(X;\R) \to B_{d-1}(X;\R)\, .
\]
\end{defn}

\begin{rem} The operator $\partial_{E,W}^\ast$ is just the formal adjoint
to the boundary operator $\partial$ with respect to the modified inner products
$\langle {-},{-}\rangle_E$ and $\langle {-},{-}\rangle_W$.

The case $E=0$ was considered 
in \cite{Catanzaro:Kirchhoff} (in this instance we simplify notation 
and write $\mathcal L^{W}$ for $\mathcal L^{E,W}$; similarly, if $W=0$ we write $\mathcal L^{E}$).
\end{rem}

Recall from \cite[defn.~1.2]{Catanzaro:Kirchhoff} that a subcomplex $T\subset X$ is a {\it spanning tree} (in dimension $d \le \dim X$) if
\begin{itemize}
\item $H_d(T) = 0$,
\item $\beta_{d-1}(T) = \beta_{d-1}(X)$, and
\item $X^{(d-1)} \subset T\subset X^{(d)}$.
\end{itemize}

\begin{rem}  
The above definition of spanning tree
is a slight generalization of the notion given
in \cite{Catanzaro:Kirchhoff} in that we do not assume $d = \dim X$. The definition
given here is a matter of convenience only and doesn't change any results of that paper
since the spanning trees in dimension $d$ of $X$ are precisely the spanning trees
of $X^{(d)}$.
When $\dim X = d = 1$, the definition is equivalent to the 
usual notion of spanning tree of a graph.
\end{rem}

For a finite complex $Y$ and a choice of degree $d\le \dim Y$, let
\begin{equation} \label{eqn:thetaT}
\theta_Y \in \Bbb N\, ,
\end{equation}
be the order of the torsion subgroup of $H_{d-1}(Y;\Bbb Z)$.
Define the {\it weight} of a spanning tree $T$ to be
\begin{equation} \label{eqn:weight-w}
w_T = w_T(W) := \theta_T^2 \prod_{b\in T_d} e^{-\beta W_b}\, .
\end{equation}

\begin{bigthm}[Improved Higher Weighted Matrix-Tree
Theorem]\label{bigthm:improved} For a finite connected CW complex $X$ and
a degree $d \le \dim X$, we have
\[
\det\mathcal L^{E,W} = \tfrac{1}{\theta^2_X}(\sum_L \tau_L)(\sum_T w_T)\, ,
\]
where the first sum is indexed over all $(d-1)$-dimensional spanning co-trees,
the second sum is indexed over all $d$-dimensional spanning trees, 
with $w_T$ is as in \eqref{eqn:weight-w}
and  $\tau_L$ as
in \eqref{eqn:tauL}.
\end{bigthm}

\begin{rem} Theorem \ref{bigthm:improved} effectively
places spanning co-trees on the same footing as spanning trees.
The special case $E = 0$ recovers \cite[thm.~C]{Catanzaro:Kirchhoff}, where the number $\mu_X$ appearing there is identified here with  $\sum_L \tau_L$.
\end{rem}

It is worth singling out the special case $E=0=W$:

\begin{bigcor}[Improved Higher Matrix-Tree
Theorem] \label{bigcor:improved}
\[
\det\mathcal L = \tfrac{1}{\theta^2_X}(\sum_L a_L^2)(\sum_T \theta_T^2)\, ,
\]
where $\mathcal L = \partial \partial^\ast\: B_{d-1}(X;\R)\to B_{d-1}(X;\R)$,
$a_L$ is as in \eqref{eqn:aL} and $\theta_T$ is as in \eqref{eqn:thetaT}.
\end{bigcor}

\begin{rem} If
the Betti numbers $\beta_j(X)$ are trivial for $j=d-1,d-2$, 
then Corollary  \ref{bigcor:improved} reduces to
the main result of \cite{DKM1}. The reduction
 follows directly from the identity
\[
a_L = \frac{\theta_{d-1}(X) \theta_{d-2}(L)}{\theta_{d-2}(X)}\, ,
\]
where $\theta_j(Y)$ denotes the order of the torsion subgroup of
$H_j(Y;\Bbb Z)$ (note: $\theta_{d-1}(Y)$ is $\theta_Y$ appearing in \eqref{eqn:thetaT}). The identity is an easy consequence of the short exact sequence
of finite abelian groups
\begin{equation*}
0 \to H_{d-1}(X)\to H_{d-1}(X,L)
\to H_{d-2}(L) \to H_{d-2}(X) \to 0\, ,
\end{equation*}
where homology is taken with integer coefficents. Here, we have used the fact that
$H_{d-1}(L) = 0$, since $\beta_{d-1}(L) = \beta_{d-1}(X)= 0$ and 
$\dim L \le d-1$.
\end{rem}

\begin{out} In \S2 we develop the elementary properties of spanning co-trees. \S3 contains the proof of the Boltzmann Splitting Formula (Theorem \ref{thm:Hodgeprob}).
 In \S4 we introduce  ``tree-co-tree''
duality which describes a bijection  between the spanning trees in a chain complex
with the spanning co-trees in the dual cochain complex. In \S5 we prove 
the Improved Higher Weighted Matrix-Tree Theorem (Theorem \ref{bigthm:improved}). 
In the appendix (\S6) we prove a summation formula 
for the pseudo-inverse to the boundary operator which unifies both the 
Kirchhoff projection and Boltzmann splitting formulas.
\end{out}

\begin{acks} We are indebted to a referee for valuable suggestions leading to a vastly improved version of the paper. Much of our writing
 was done while the last author was visiting the University of Copenhagen. He is indebted to Lars Hesselholt for providing him with support from the Niels Bohr Professorship.
 
This material is based upon work supported by the National
Science Foundation Grant  CHE-1111350
and the Simons Foundation Collaboration Grant  317496.
\end{acks}

\section{Spanning co-trees \label{sec:cotrees}}

If $\Bbb F$ is a field,
recall that a $k$-chain $c\in C_k(X;\Bbb F)$ is any  $\Bbb F$-linear 
combination of $k$-cells. If $b\in X_k$ is a $k$-cell, we write $\langle c,b \rangle$
for the coefficient of $b$ appearing in $c$. If $\langle c,b \rangle \ne 0$, we say
that $b$ {\it appears} in $c$. 

\begin{defn} 
  A $k$-cell $b \in X_k$ is said to be {\em essential} if there
  exists a $k$-cycle $z \in Z_k(X;\Q)$ such that $\langle z,b \rangle
  \neq 0$.
\end{defn}

\begin{lem}[{\cite[lem.~2.2]{Catanzaro:Kirchhoff}}]\label{lem:add-or-rem}
  Adding or removing an essential $d$-cell from $X$ increases or 
  decreases $\beta_{d}(X)$ by one, respectively, and fixes $\beta_{d-1}(X)$.
\end{lem}

\begin{lem} \label{lem:co-tree-exist}
  $X$ has a spanning co-tree.
\end{lem}

\begin{proof} The homomorphism $H_{d-1}(X^{(d-1)};\Q) \to H_{d-1}(X;\Q)$ is surjective
with kernel $K_1 := B_{d-1}(X;\Q)$. Set $Y^1 := X^{(d-1)}$.
Suppose that $c \in  B_{d-1}(X;\Q)$ is nontrivial.
Let $b$ be a $(d-1)$-cell of $X$ that appears in $c$.
Let $Y^2$ be the result of removing $b$ from $X^{(d-1)}$. The homomorphism
$H_{d-1}(Y^2;\Q) \to H_{d-1}(X;\Q)$ is surjective; let $K^2$ be
its kernel. Then the rank of $K^2$ is strictly less than that of $K^1$ by 
Lemma~\ref{lem:add-or-rem}.
Furthermore,  $\beta_{d-2}(Y^2) = \beta_{d-2}(X)$. By iterating
(with $Y^2$ replacing $Y^1$, etc.)  we eventually obtain a subcomplex $Y^k  \subset X^{(d-1)}$ such that 
$H_{d-1}(Y^k;\Q) \to H_{d-1}(X;\Q)$ is an isomorphism. Then $Y^k$ is a spanning co-tree.
\end{proof}


\begin{prop}\label{prop:co-tree-equiv} Let $\mathbb F$ be a field of characteristic zero.
  Let $L \subset X$ be a $(d-1)$-dimensional subcomplex that contains $X^{(d-2)}$. Then
  $L$ is a spanning co-tree if and only if the composition
  \begin{equation} \label{eqn:iso}
    C_{d-1}(L;\mathbb F) \ra C_{d-1}(X;\mathbb F) \ra C_{d-1}(X;\Bbb F)/B_{d-1}(X;\mathbb F) 
  \end{equation}
  is an isomorphism.
\end{prop}

\begin{proof} Clearly, it suffices to prove the assertion when $\mathbb F = \Q$.
Suppose $L$ is such that \eqref{eqn:iso} is an isomorphism. Consider the following commutative diagram:
  \[
    \xymatrix{
      Z_{d-1}(L;\Q)
      \ar@{->}[r]^{i_L}
      \ar@{->}[d]_k
      &
      Z_{d-1}(X;\Q) 
      \ar@{->}[r]^p
      \ar@{->}[d]
      &
      H_{d-1}(X;\Q)
      \ar@{->}[d]^j
      \\
      C_{d-1}(L;\Q)
      \ar@{->}[r]_a
      &
      C_{d-1}(X;\Q)
      \ar@{->}[r]_(0.35){\pi}
      &
      C_{d-1}(X;\Q)/B_{d-1}(X;\Q)\, .
    }
  \]
  The left square is a pullback and the right square is a
  pushout. By assumption, the bottom composite is 
  an isomorphism, so the top composite is also an isomorphism.
Therefore, $i_{L*}\:H_{d-1}(L;\Q) \ra H_{d-1}(X;\Q)$ is an isomorphism.
  The remaining two conditions of Definition \ref{defn:spanning} are easily verified.
Consequently, $L$ is a spanning co-tree.

For the converse, let $x \in C_{d-1}(L;\Q)$ be such that 
  $(\pi \circ a)(x) = 0$. Then $a(x) \in B_{d-1}(X;\Q) \subset Z_{d-1}(X;\Q)$.
Since the left square is a pullback, we infer that $x\in Z_{d-1}(L;\Q)$.
  But $p \circ i_L$ is an isomorphism, and $j$ is injective, so $x=0$.
  This establishes the injectivity of \eqref{eqn:iso}.

  For surjectivity, let $z \in C_{d-1}(X;\Q)/B_{d-1}(X;\Q)$. Lift this to any
   element $y \in C_{d-1}(X;\Q)$.
  Then $\partial(y) \in C_{d-2}(L;\Q) = C_{d-2}(X;\Q)$ lies in $Z_{d-2}(L;\Q)$ since
  $\partial^2=0$. Furthermore, the pushforward
  of the homology class $[\partial(y)] \in H_{d-2}(L;\Q)$ in $H_{d-2}(X;\Q)$
  is trivial, since $H_{d-2}(L;\Q) \cong
  H_{d-2}(X;\Q)$. It follows that $\partial (y)$ lies in $B_{d-2}(X;\Q) = B_{d-2}(L;\Q)$.
  Hence, 
  $\partial(y) = \partial(x)$ for some $x \in C_{d-1}(L;\Q)$. Then
  $a(x) - y$ lies in $Z_{d-1}(X;\Q)$, and since $L$ is a spanning co-tree,
  there exists $x' \in Z_{d-1}(L;\Q)$ so that 
  $\pi(a(x)-y) = (j\circ p\circ i_L)(x')$. But $z = \pi(y)$, so 
    \[
    z = \pi(y) = \pi(a(x)) - j(p(i_L(x'))) = \pi(a(x - k(x')))\, .
  \]
  We conclude that \eqref{eqn:iso} is surjective.
\end{proof}

\begin{rem} A referee has pointed out to us that
Lemma \ref{lem:co-tree-exist} as well as 
Proposition \ref{prop:co-tree-equiv} admit alternative proofs using matroids.  
\end{rem}

\begin{lem}\label{lem:splitting} Let $\mathbb F$ be a field.
  Then a splitting of the quotient homomorphism 
  $C_{d-1}(X;\mathbb F) \ra C_{d-1}(X;\mathbb F)/B_{d-1}(X;\mathbb F)$ induces by restriction a 
  splitting of the quotient homomorphism
  $Z_{d-1}(X;\mathbb F) \ra H_{d-1}(X;\mathbb F)$.
\end{lem}

\begin{proof}
  Consider the following commutative diagram, with exact rows.
  \[
    \xymatrix{
      0 
      \ar@{->}[r]
      &
      B_{d-1}(X;\mathbb F)
      \ar@{->}[r]
      \ar@{=}[d]
      &
      Z_{d-1}(X;\mathbb F)
      \ar@{->}[r]^p
      \ar@{->}[d]
      &
      H_{d-1}(X;\mathbb F)
      \ar@{->}[r]
      \ar@{->}[d]
      &
      0
      \\
      0
      \ar@{->}[r]
      &
      B_{d-1}(X;\mathbb F)
      \ar@{->}[r]
      &
      C_{d-1}(X;\mathbb F)
      \ar@{->}[r]_(0.35){\pi}
      &
      C_{d-1}(X;\mathbb F) / B_{d-1}(X;\mathbb F)
      \ar@{->}[r]
      &
      0\, .
    }
  \]
  Since $H_{d-1}(X;\mathbb F) \subset C_{d-1}(X;\mathbb F)/B_{d-1}(X;\mathbb F)$, we can restrict the
  given splitting to get a map $H_{d-1}(X;\mathbb F) \ra C_{d-1}(X;\mathbb F)$. A simple
  diagram chase shows that this map factors through $Z_{d-1}(X;\mathbb F)$.
  \end{proof}

\begin{notation} 
For $i \le j$ and $Y \subset X$ be a subcomplex.
Set 
\[
Y_{j,i} = Y^{(j)}/Y^{(i)}\, .
\] 
Then $Y_{j,i} \subset X_{j,i}$ is a subcomplex of dimension $\le j$.
\end{notation}

\begin{cor} \label{cor:co-tree-equiv} 
Assume $d \ge 2$. Then the operation $L \mapsto L_{d,d-2}$ 
defines a bijection between
the spanning co-trees of $X$ and the spanning co-trees of $X_{d,d-2}$. Furthermore, $
a_{L_{d,d-2}} = a_L\, .
$
\end{cor}

\begin{proof}  By definition $C_{d-1}(X_{d,d-2};\Bbb F) = Z_{d-1}(X;\Bbb F)$, so the diagram
\[
C_{d-1}(L;\Bbb F) \to C_{d-1}(X;\Bbb F)  \to C_{d-1}(X;\Bbb F)/B_{d-1}(X;\Bbb F)  
\]
coincides with the diagram
\[
Z_{d-1}(L_{d,d-2};\Bbb F) \to Z_{d-1}(X_{d,d-2};\Bbb F)  \to H_{d-1}(X_{d,d-2};\Bbb F) \, .
\]
It follows that
the first part amounts to a restatement of Proposition \ref{prop:co-tree-equiv}.

To prove the second part, use the homotopy pushout diagram
\[
\xymatrix{
L \ar[r]\ar[d] & X \ar[d] \\
L_{d,d-2}\ar[r] & X_{d,d-2}
}
\]
and the long exact sequences in homology 
associated with the horizontal maps. Using the five-lemma
we infer that the homomorphism $H_\ast(X,L) \to H_\ast(X_{d,d-2},L_{d,d-2})$ is an isomorphism in all degrees.
\end{proof}

\begin{rem} \label{rem:co-tree-equiv}  Corollary \ref{cor:co-tree-equiv} reduces the proof of 
Theorem \ref{thm:Hodgeprob} to the special case of CW complexes $X$ having trivial 
$(d-2)$-skeleton (if $d \le 1$, triviality is automatic).
For such complexes the number $a_{L}$ coincides with the order of 
cokernel of the
homomorphism
\[
C_{d-1}(L;\Bbb Z) \to C_{d-1}(X;\Bbb Z)/B_{d-1}(X;\Bbb Z) \, ,
\]
The triviality of the $(d-2)$-skeleton implies
that the displayed map coincides with the homomorphism $H_{d-1}(L;\Bbb Z) \to 
H_{d-1}(X;\Bbb Z)$.

Let $H_{d-1}(X;\Bbb Z)_{\text{tor}}\subset H_{d-1}(X;\Bbb Z)$ be the torsion subgroup.
Let $b_L$ denote the order of the cokernel of the composite map 
\begin{equation}\label{eqn:torsion-free}
H_{d-1}(L;\Bbb Z) @>>> H_{d-1}(X;\Bbb Z) @>>> H_{d-1}(X;\Bbb Z)/H_{d-1}(X;\Bbb Z)_{\text{tor}}\, .
\end{equation}
Then \eqref{eqn:torsion-free} is a monomorphism of finitely generated free abelian groups with finite cokernel. Up to sign, the determinant of \eqref{eqn:torsion-free} is well-defined
and coincides with $b_L$ (see  \cite[prop.~6.63]{Catanzaro:Kirchhoff}). 
Furthermore, 
\begin{equation}
a_L = \theta_X b_L \, ,
\end{equation} 
where $\theta_X$ is the order of $H_{d-1}(X;\Bbb Z)_{\text{tor}}$. 
\end{rem}

To complete the proof of Theorem \ref{thm:Hodgeprob}, we  will construct a splitting 
of the map $C_{d-1}(X;\R) \to C_{d-1}(X;\R)/B_{d-1}(X;\R)$ that will give the relevant summation formula.
For this, we shall appeal to the theory of generalized inverses.

\section{The proof of Theorem \ref{thm:Hodgeprob}}

\subsection{Generalized Inverses} 
The theory of generalized inverses was developed
to study linear systems $Ax=b$ for which $A^{-1}$ does not exist. 
Let $A$ be an $m \times n$ matrix over $\R$, and let $b \in \R^m$ be given. Consider the
linear system $Ax=b$. In general, such systems need not have a
(unique) solution.  One way to study the system is to 
attempt to minimize the norm of the residual vector
$Ax-b$. Among all such $x$ for which the norm of $Ax-b$
is minimizing, we impose the additional constraint that the norm of $x$
is minimizing.  This is called a {\it least squares problem}.\footnote{This is slightly more general than the usual formulation. The classical least squares problem assumes that 
$A$ is injective. We will be primarily concerned here with the case when $A$ is surjective.} 

\begin{rem} When $A$ is surjective the residual vector having minimum norm 
is the zero vector.
In this case the least squares problem reduces to the problem of finding a solution
of $Ax=b$ such that the norm of $x$ is minimized. 
\end{rem}

The Moore-Penrose pseudo-inverse  $A^+$
gives a preferred solution to the least squares problem.
If $b$ lies in the image of $A$, then a solution to $Ax=b$ exists
and the Moore-Penrose solution $A^+b$ will be a solution having
the smallest norm. 
Furthermore, the matrix $A^+$ exists and is unique \cite{Penrose:general}, \cite[p.~109]{Ben-Israel:GenInv}.

The operation $A \mapsto A^+$ satisfies the identities
\begin{equation} \label{eqn:MP-inverse}
A^+ = A^{\ast}(AA^{\ast})^+ = (A^{\ast}A)^+A^{\ast}\, ,
\end{equation}
 where $A^\ast$ is the transpose of $A$ (cf.\ \cite[chap.~1.6, ex.~18(d)]{Ben-Israel:GenInv}). In particular, when $A$ is surjective,
we obtain the formula 
\begin{equation} \label{eqn:MP-inverse-surject}
A^+ = A^{\ast}(AA^{\ast})^{-1}\, .
\end{equation} 

\begin{rem} \label{rem:surject} If $A$ is surjective, then one may drop the requirement that the target of $A$ is based. 
That is, suppose more generally that $A\: \R^n \to V$ is a surjective linear transformation
where $V$ is not necessarily based.
Then the least squares problem as well as the formula \eqref{eqn:MP-inverse-surject} make sense if we use
the formal adjoint $A^\ast\: V^\ast \to  (\R^n)^\ast = \R^n$ in place of the transpose.
Similarly, if $A$ is injective, we may drop the requirement that the source of $A$ is based.
\end{rem}

We will need a weighted version of the least squares problem. For this, we
weight the standard basis elements $\{e_i\}_{i=1}^n$ of $\R^n$ by means of a positive 
functional $\mu\: \{e_i\}_{i=1}^n \to \mathbb R_+$. Then $\mu$ defines
a modified inner product $\langle {-},{-}\rangle_\mu$ on $\R^n$, determined by
$\langle e_i,e_j\rangle_\mu := \mu(e_i)\delta_{ij}$. The {\it weighted least squares problem} is
to minimize 
$|Ax-b|$ such that $|x|_\mu$ is also minimized.
Again, the solution $x = A^+b$ exists
and is unique, where now $A^+$
is the weighted version of the Moore-Penrose pseudo-inverse.

Assume now that $A$ has rank $m$, i.e., $A$ is surjective.
For a subset $S\subset \{1,\dots,n\}$ of cardinality
$m$,
let $A_S$ be the $m\times m$ submatrix whose rows correspond to 
indices in $S$:
\[
  (A_S)_{ij} := A_{ij}\, , \,\,\,\, \mbox{ for } 
  i=1,\ldots m, \,\, j \in S\, .
\]
We will consider only those $S$ such that $A_S$ is invertible. 
Let $i_S\: \R^m \to \R^n$ denote the inclusion given by the rows corresponding
to $S$. Set
\[
t_S := \det(A_S)^2 \prod_{i \in S} \frac{1}{\mu(e_i)}\, ,
\] 
and set $\nabla := \sum_S t_S$.
We can now state the summation formula for $A^+$ in the case of surjective $A$.

\begin{thm}[cf.~{\cite[Thm 1]{Berg:Three}, \cite[th.~2.1]{Ben-Tal:GeomProp}}]
  \label{thm:MPformula} Let $A$ be an $m\times n$ matrix of rank $m$ defined over $\R$.
  Then the weighted Moore-Penrose pseudo-inverse of $A$ is given by
  \[
    A^+ = \tfrac{1}{\nabla} \sum_S t_S i_S(A_S)^{-1}\, ,
  \]
  where the sum is taken over all indices $S \subset \{1,2,\ldots,n\}$ such that
  $A_S$ is invertible.
\end{thm}

\begin{rem} Theorem 2.1 of \cite{Ben-Tal:GeomProp} concerns the case
when $A$ has rank $n$, i.e., when $A$ is injective with arbitrary weights. 
The main result of \cite{Berg:Three} applies to general $A$ in the unweighted case 
$\mu = 1$. When $A$ is surjective, it is  straightforward to deduce the weighted case from the unweighted one by the following transformation: replace
$A$ by $\hat A = AM^{-1}$, where $M$ is the diagonal matrix having entries $\sqrt{\mu(e_i)}$ and replace $x$ by $\hat x = Mx$. This converts the weighted least squares problem
$(Ax=b,\mu)$ to an equivalent unweighted problem 
$(\hat A \hat x = b,1)$. The formula
displayed in Theorem \ref{thm:MPformula} is easily deduced from this, and we will omit the argument. 
\end{rem}

\begin{rem} \label{rem:MPformula} 
Suppose that $A\: \R^n \to V$ is a surjective linear transformation
in which $V$ is not necessarily based (cf.\ Remark \ref{rem:surject}).
Furthermore, suppose that $H \subset V$ is a lattice, i.e., a 
finitely generated abelian subgroup such that the induced
map $H \otimes_{\Bbb Z} \Bbb R \to V$ is an isomorphism. Then a choice of basis
for $H$ determines one for $V$ and
Theorem \ref{thm:MPformula} applies.

Moreover,
the formula is invariant with respect to 
base changes for $H$, since the numbers
$t_S$ are squares of determinants. 
Consequently, Theorem \ref{thm:MPformula}
is really a statement about surjective linear transformations $A\: \R^n \to V$ 
for vector spaces  $V$ that come equipped with a preferred lattice $H$.
\end{rem}

\begin{proof}[Proof of Theorem~\ref{thm:Hodgeprob}]
By Remark \ref{rem:co-tree-equiv}  
there is no loss in assuming that $X$ has trivial $(d-2)$-skeleton.
  By Remark~\ref{rem:splitting} and Lemma~\ref{lem:splitting}, it suffices to
  produce a splitting of the quotient homomorphism
  $\pi\: C_{d-1}(X;\R) \ra C_{d-1}(X;\R) /
  B_{d-1}(X;\R)$. Here we use the weighted basis of $C_{d-1}(X;\R)$ defined by 
  the cells and the weighting given by $b \mapsto e^{\beta E_b}$.
  Use the lattice $H \subset C_{d-1}(X;\R) /
  B_{d-1}(X;\R)$ given by the image of the homomorphism
  $C_{d-1}(X;\Z) / B_{d-1}(X;\Z) \to C_{d-1}(X;\R) /
  B_{d-1}(X;\R)$.
  
  Applying 
  Theorem~\ref{thm:MPformula} and Remark \ref{rem:MPformula}   to $\pi$ gives a splitting, written as a sum over
  subsets $S$ of the basis elements of $C_{d-1}(X;\R)$.  By
  Proposition~\ref{prop:co-tree-equiv}, the collection of these 
  subsets are in bijection with
  the set of spanning
  co-trees. The inclusion $i_S$ corresponds to the inclusion 
  $C_{d-1}(L;\R) \ra
  C_{d-1}(X;\R)$ and $\phi_L$ corresponds to $A_S$. Then the determinant of $A_S$
  is given up to sign by $\tau'_L := \tau_L/\theta_X$ (cf.~ Remark \ref{rem:co-tree-equiv})
  and the prefactor is given by the reciprocal of $\nabla' := 
  \sum_L \tau'_L$. Hence, the desired splitting is given by
  \[
  \tfrac{1}{\nabla'} \sum \tau'_L\psi_L = \tfrac{1}{\nabla} \sum \tau_L\psi_L \, .\qedhere
   \]
\end{proof}

\begin{rem}\label{rem:oldKthm} If we fix a  function $W\: X_d \to \R$,
  we may instead apply \cite[Thm 1]{Berg:Three} to the inclusion map 
  $q\: Z_{d}(X;\R) \to C_{d}(X;\R)$.  
  This produces an orthogonal splitting $C_d(X;\R) \to Z_d(X;\R)$ to $q$ in the modified inner product on $C_{d}(X;\R)$. The splitting is written as a sum indexed over the set of
   spanning trees as
  in~\cite{Catanzaro:Kirchhoff}. In fact, this
gives quick alternative proofs to Theorem A and Addendum 
  B in~\cite{Catanzaro:Kirchhoff}.
\end{rem}

\section{Tree-co-tree duality}
\label{sec:tree-co-tree}

Let $X$ be a finite connected CW-complex and fix a natural number
$d \le \dim X$. We
let 
\[
(C_\ast(X),\partial) \quad \text{and} \quad (C^\ast(X),\delta)
\]
denote  the cellular chain and cochain complexes of $X$ over $\Z$.
The set $X_d$  determines preferred isomorphism $C^d(X) \cong C_d(X)$. With this identification, the coboundary operator $\delta\: C^{d-1}(X) \to C^d(X) $ corresponds to $\partial^\ast\: C_{d-1}(X) \to C_d(X)$, the formal adjoint to $\partial$ (here,
coefficients can be taken in any commutative ring).

Denote by ${\mathcal F}_{d}(X)$ and ${\mathcal F}_{d}^{*}(X)$ the sets of $d$-dimensional spanning trees and spanning co-trees, respectively, of $X$. 
For $i \le j$, recall the subquotient $X_{j,i} = X^{(j)}/X^{(i)}$.

\begin{lem} \label{lem:reduction-bijection-one}
There are preferred bijections  
\[
{\mathcal F}_{d}(X) \cong {\mathcal F}_{d}(X_{d,d-2}) \qquad \text{and} \qquad
{\mathcal F}_{d-1}^{*}(X) \cong {\mathcal F}_{d-1}^{*}(X_{d,d-2})\, .
\] 
\end{lem}

\begin{proof} The second bijection is merely a restatement the first part of Corollary \ref{cor:co-tree-equiv}. 
For  $T\in  {\mathcal F}_{d}(X)$ and  $L\in{\mathcal F}_{d-1}^{*}(X)$,
the assignments
\[
T\mapsto T_{d,d-2} \qquad \text{and} \qquad L\mapsto L_{d,d-2}
\]
define maps ${\mathcal F}_{d}(X)\to  {\mathcal F}_{d}(X_{d,d-2})$
and $ {\mathcal F}_{d-1}^{*}(X) \to {\mathcal F}_{d-1}^{*}(X_{d,d-2})$.  
Inverse maps are defined as follows:
given a spanning tree $T'$ for $X_{d,d-2}$, we form a complex $T$ by attaching the set of $d$-cells  appearing in $T'$
to $X^{(d-1)}$. It is straightforward to verify that $T$ is a $d$-dimensional spanning tree for $X$. 

Similarly, if $L'$ is a spanning co-tree for $X_{d,d-2}$, then the complex
$L$ given by attaching the $(d-1)$-cells of $L'$ to $X^{(d-2)}$ gives a $d$-dimensional spanning co-tree.
\end{proof}

 For $T \in {\mathcal F}_{d}(X)$, let $\theta_{T}$ be the order of the torsion subgroup of $H_{d-1}(T; \mathbb{Z})$ (cf. \cite[p.~3]{Catanzaro:Kirchhoff}) and for $L \in {\mathcal F}_{d-1}^{*}(X)$, recall  that $a_L$ is the order of $H_{d-1}(X,L; \mathbb{Z})$.
 
\begin{lem} \label{lem:reduction-bijection-two} 
For $T \in {\mathcal F}_{d}(X)$ and $L \in {\mathcal F}_{d-1}^{\ast}(X)$ we have
\[
\theta_T =\theta_{T_{d,d-2}} \qquad \text{and} \qquad a_L =a_{L_{d,d-2}} \, .
\]
\end{lem}

\begin{proof} The first part follows immediately from the exactness of the sequence
\[
0 \to H_{d-1}(T;\Bbb Z) \to H_{d-1}(T_{d,d-2};\Bbb Z) \to H_{d-2}(T^{(d-2)};\Bbb Z)
\]
and the fact that $H_{d-2}(T^{(d-2)};\Bbb Z)$ is free abelian.
The second part  is just
a restatement of the second part of Corollary \ref{cor:co-tree-equiv}. 
\end{proof} 
 
 \begin{rem}\label{rem:reduction}
In view of these lemmas, all the relevant properties of $d$-dimensional spanning trees and $(d-1)$-dimensional spanning co-trees depend only on the boundary operator $\partial: C_{d}(X) \to C_{d-1}(X)$ and are therefore reduced to properties of two-stage chain complexes of finitely generated free abelian groups, i.e., to statements in linear algebra over $\Z$. We will now make this precise.\end{rem}

For a finite set $S$, let
$\Bbb Z^S$ be the free abelian group with basis set $S$.

\begin{defn} For finite sets $P$ and $Q$,
set $A = \Bbb Z^P$ and $B = \Bbb Z^Q$. 
Let  $\partial: A \rightarrow B$ be a homomorphism.
A {\it spanning tree} of $\partial$ consists of a subset $T \subset P$ such that the composition
\[
\Bbb Z^{T} @> \subset >> 
A @> \partial >> B
\]
is an isomorphism modulo torsion. 
The set
of spanning trees of $\partial $ is denoted by $\mathcal F(\partial)$.

Similarly, a {\it spanning co-tree} of $\partial$ is a subset
$L\subset Q$ such that the composition
\begin{equation} \label{eqn:co-tree-f-defn}
\Bbb Z^L \subset B @>>> B/\partial(A)
\end{equation}
is an isomorphism modulo torsion. 
The set of spanning co-trees of $\partial$ is denoted by $\mathcal F^\ast(\partial)$. 
\end{defn}

For an abelian group $U$ let $U^\ast := \hom_\Z(U,\Z)$. 
This defines an contravariant endo-functor on abelian groups. 
Let $\partial^\ast\: B^\ast \to A^\ast$ be the homomorphism
induced by $\partial$. For a subset $S\subset P$, let $S^\perp = P \setminus S$ denote its complement.
The proof of the following is a straightforward exercise left to the reader.

\begin{lem}[Tree-Co-tree Duality]\label{lem:tree-co-tree}
The operation $T \mapsto T^{\perp}$ induces a bijection
$
\mathcal F(\partial)\cong  \mathcal F^\ast(\partial^\ast)
$.
\end{lem}

\subsection{Finite chain complexes}
A $\Z$-graded chain complex $C_\ast$ over $\Z$ is {\it finite}
if it is degree-wise finitely generated and free and has finitely many non-zero terms.
If $C_\ast$ is finite then so is its linear dual
\[
DC_\ast = \hom(C_{-\ast},\Bbb Z)\, .
\]
If we fix $d\in \Z$, then the notion of $d$-dimensional spanning tree and 
co-tree is defined in this context. Let $\mathcal F_d(C_\ast)$ be the set of $d$-dimensional spanning trees of $C_\ast$. Similarly, we let $\mathcal F^\ast_d(C_\ast)$ be the set of $d$-dimensional spanning co-trees of $C_\ast$. The following is an immediate consequence
of Lemma \ref{lem:tree-co-tree} combined with Remark \ref{rem:reduction}.

\begin{cor}[Chain Tree-Co-tree Duality] 
\label{chain-tree-cotree} Assume $C_\ast$ is finite.  Then for $d\in \Z$, 
there is a preferred bijection
$
\mathcal F_d(C_\ast) \cong \mathcal F_{-d}^\ast(DC_\ast)
$.
\end{cor}

\begin{rem} It is not difficult to 
give a spectrum-level version of Corollary 
\ref{chain-tree-cotree}. 
When $X$ is a finite CW spectrum, the 
spectrum $DX$ (corresponding to the linear dual
of a chain complex) is the Spanier-Whitehead dual of $X$, i.e,
the function spectrum $F(X,S^0)$.

Corollary \ref{chain-tree-cotree} may
be related to some of the results of \cite{MMRW}.
\end{rem}

\section{The proof of Theorem \ref{bigthm:improved}}
\label{sec:matrix-tree-cotree}

Given a scalar structure on $X$ and a dimension $d\le \dim X$, one has a pair of biased Laplacians defined by the commutative diagrams
\begin{equation}
\label{define-L}
\xymatrix{ 
B_{d-1}(X;\R)\ar[r]^(.6){\partial^{*}e^{\beta E}}
\ar[dr]_{{\mathcal L}^{E,W}} & B^{d}(X;\R) \ar[d]^{\partial e^{-\beta W}} \\
& B_{d-1}(X;\R)
} \qquad
\xymatrix{         & B^{d}(X;\R) \ar[dl]_{{\mathcal L}^\ast_{W,E}} \ar[d]^{\partial e^{-\beta W}} \\
          B^{d}(X;\R) & B_{d-1}(X;\R)\ar[l]^{\partial^{*}e^{\beta E}}\, .
}
\end{equation}
Observe that 
\[
({\mathcal L}^{E,W})^\ast = {\mathcal L}^{\ast}_{W,E}\, .
 \] 

The operators  ${\mathcal L}^{E,W}$ and ${\mathcal L}^{\ast}_{W,E}$ are invertible and have the same determinant. To avoid notational clutter, when $E$ and $W$ are understood we
simplify notation and set ${\mathcal L} := {\mathcal L}^{E,W}$
and ${\mathcal L}^{*} : = {\mathcal L}^{\ast}_{W,E}$.

Recall that the goal is
to  exhibit a decomposition of the determinant of $\mathcal L$ as a sum over trees and co-trees:
\begin{equation}
\label{matrix-tree-cotree} 
\det({\mathcal L}) = \det({\mathcal L}^{*})
=\tfrac{1}{\theta_X^2}\sum_{L,T}\tau_{L}w_{T}\, .
\end{equation}

We start with a weak version that establishes \eqref{matrix-tree-cotree} up to a factor that does not depend on either $E$ or $W$. We consider
 $\mathcal L$ as a function of the pair $(E,W)$ taking values in the space of linear operators
 on $B_{d-1}(X;\R)$.  We then measure the variation of $\ln\det({\mathcal L})$
 in $W$  assuming $E$ is held fixed.  This gives
\begin{equation} \label{matrix-tree-cotree-weak} 
\begin{alignedat}{1}
d\ln\det({\mathcal L}) &= -{\rm Tr}(\partial dW e^{-W} \partial^{*} e^{E} {\mathcal L}^{-1}) \\
& = -{\rm Tr}(dW e^{-W} \partial^{*} e^{E} {\mathcal L}^{-1} \partial) \\
&= -{\rm Tr}(dW A)\, ,
\end{alignedat}
\end{equation}
where
\begin{equation}
\label{define-A} A := e^{-W} \partial^{*} e^{E} {\mathcal L}^{-1}\:  B_{d-1}(X;\R) \rightarrow C_{d}(X;\R)\, 
\end{equation}
(compare \cite[eqns.~(6),(7)]{Catanzaro:Kirchhoff}).
It is straightforward to check that $A$ is the pseudo-inverse for $\partial: C_{d}(X) \to B_{d-1}(X)$. Thus we can follow, {\it mutatis mudandis}, the proof of  \cite[prop.~4.2]{Catanzaro:Kirchhoff} (which is the $E = 0$ particular case of $\mathcal L$) to arrive at
\begin{eqnarray}
\label{matrix-tree-cotree-weak-1} \det({\mathcal L}) = \det({\mathcal L}^{*}) = \bar{\gamma}(E)\sum_{T}w_{T}(W)\, .
\end{eqnarray}
where the prefactor $\bar\gamma(E)$ is independent of $W$.

We next make use of duality, i.e.,
$\det(\mathcal {L}) = \det(\mathcal {L}^{*}) = \partial^{*} e^{E} \partial e^{-W}$, so that variation over $E$, with $W$ held fixed, is equivalent to the case considered above. Consequently,
\begin{eqnarray}
\label{matrix-tree-cotree-weak-2} \det({\mathcal L}) = \det({\mathcal L}^{*}) = \bar{\gamma}^{\ast}(W)\sum_{L}\tau_{L}(E)\, , 
\end{eqnarray}
where the prefactor $\bar \gamma^\ast(W)$ is independent of $E$.

Setting  \eqref{matrix-tree-cotree-weak-1} and \eqref{matrix-tree-cotree-weak-2} equal, we immediately obtain the  following weak form of Theorem \ref{bigthm:improved}:
\begin{eqnarray}
\label{matrix-tree-cotree-3} \det({\mathcal L}) = \det({\mathcal L}^{*}) = \gamma(\sum_{L}\tau_{L})(\sum_{T} w_{T})\, ,
\end{eqnarray}
in which the prefactor $\gamma$ is a constant independent of $E$ and $W$. 
The proof of Theorem \ref{bigthm:improved} is complete once we
 establish the following claim.
 
 \begin{cl} $\gamma = 1/\theta_X^2$.
\end{cl}

To prove the claim, by the argument of 
\cite[\S6]{Catanzaro:Kirchhoff} it is enough
to restrict to the case when $X$ is a spanning tree of dimension $d$. Furthermore, by 
Lemmas \ref{lem:reduction-bijection-one} and \ref{lem:reduction-bijection-two} we can \ assume that $X = X_{d,d-2} = X^{(d)}/X^{(d-2)}$. 
Since $\gamma$ is independent of $E$ and $W$ we can further assume that
$E = 0 = W$. In this instance, by \cite[cor.~D]{Catanzaro:Kirchhoff}
it will suffice to establish the identity
\begin{equation} \label{eqn:mu-a-identity}
\mu_X = \sum_{L\in \mathcal F^*(X)} a_L^2\, ,
\end{equation}
where $\mu_X$ denotes the square covolume of the lattice 
$B_{d-1}(X;\Z) \subset
B_{d-1}(X;\R)$.

With respect to our hypotheses,
the cellular chain complex of $X$ is determined 
by the (rationally injective) homomorphism
$\partial\: A \to B$ of finitely generated free abelian groups.
By slight abuse of notation, denote 
this chain complex by $X$ and let $Y$ be the dual chain complex
given by $\partial^\ast\: B^\ast \to A^\ast$, where $A^\ast =\hom(A,\Z)$. Then
\begin{equation} \label{eqn:next-to-last-step}
\begin{alignedat}{1}
\det(\mathcal L) &= \det(\mathcal L^\ast)\\
& = \det(\partial^\ast \partial) \\
& =  
\tfrac{\mu_Y}{\theta_Y^2} \sum_{T \in \mathcal F(Y)} \theta^2_T\, . 
\end{alignedat}
\end{equation}
where the last equality follows from \cite[cor.~D]{Catanzaro:Kirchhoff}. 
In the above, we have implicitly identified 
$\partial^\ast$ with the transpose of $\partial$ using the preferred bases.
The number
$\mu_Y$ is the square of the covolume of the lattice defined by the image 
$\partial^\ast(B^\ast) \subset \partial^\ast(B^\ast)\otimes_{\Z} \R $
where the latter term is given an inner product by declaring it
to be an isometric subspace of $A^\ast \otimes_{\Z} \Bbb R = \hom(A,\R)$.
 By  Lemma \ref{lem:tree-co-tree}, $\mathcal F(Y) = \mathcal F^\ast(X)$.

Given a spanning co-tree $L \in  \mathcal F^\ast(\partial)$, 
let $a_L$ denote the order of the cokernel of $L \to B/\partial(A)$. For a subgroup $T\subset A$ let $\theta_T$ denote the order
of the torsion subgroup of the cokernel of the composition $T\to A \to B$. Note that
these definitions coincide with the ones we gave in the case of CW complexes with trivial $(d-2)$-skeleton.

The identity \eqref{eqn:mu-a-identity} (and hence the claim) is now an immediate consequence of the following:

\begin{lem} \label{lem:final} With respect to the above assumptions, we have
\begin{enumerate}
\item $\det(\mathcal L)=\mu_X$;
\item $\theta_Y = \theta_X$;
\item $\mu_Y = \theta_Y^2$;
\item $a_L = \theta_{(L^{\perp})^\ast}$.
\end{enumerate}
\end{lem}

\begin{proof} Statement (1) is a special case of \cite[cor.~D]{Catanzaro:Kirchhoff}).
 
Let $C$ be the cokernel $\partial$. Then we have a short exact sequence 
 $0 \to A \to B \to C\to 0$. Applying $\hom({-},\Z)$ gives a short exact sequence
 \[
 0 \to C^\ast \to B^\ast \to A^\ast \to \text{ext}(C,\Z) \to 0\, .
 \]
 Since $C$ is finitely generated, the group $\text{ext}(C,\Z)$ is a torsion group which
 is (non-canonically) isomorphic to the torsion subgroup of $C$. This gives (2).
 
Statement (3) follows from the discussion after definitions 6.2 and 6.5 of \cite{Catanzaro:Kirchhoff} as applied to the real isomorphism 
$\partial^*(B^\ast) \to A^\ast$.  
Lastly, (4) is easily deduced from the fact that the cokernels of $\Z^L \to B/\partial(A)$ and
$A \to \Z^{L^\perp }$ are canonically isomorphic.   
\end{proof}

This establishes the claim and completes the proof of Theorem \ref{bigthm:improved}.
 
 \begin{rem} There is different proof  of the claim
that makes use of the ``low-temperature limit'' method
of \cite[\S5]{Catanzaro:Kirchhoff} (with respect to $E$ and $W$) as applied to equation 
 \eqref{matrix-tree-cotree-3},
 with the sum on the right-hand side dominated by one term. 
However, the proof we have given here has the advantage of being shorter and less technical. 
 \end{rem}

\section{Appendix:  Kirchhoff-Boltzmann unification}
\label{sec:general-pseudo-inv}

For any fixed dimension $d$, 
the  boundary operator $\partial\: C_{d}(X;\R) \rightarrow C_{d-1}(X;\R)$ factors as
\[
C_{d}(X;\R) @>p >> B_{d-1}(X;\R)  @> i >> C_{d-1}(X;\R) 
\] 
in which the first map is surjection and the second is an injection. 

With respect to these inner product structures, one infers that the pseudo-inverse 
of $\partial \: C_{d}(X;\R) \to C_{d-1}(X;\R)$ is given by the composition 
\[
C_{d-1}(X;\R) @> i_E^+ >>  B_{d-1}(X;\R) @> p_W^+ >> C_{d}(X;\R)
\]
consisting of the pseudo-inverses of $i$ and $p$.
(cf. \cite[p.~48, ex.~17]{Ben-Israel:GenInv}).
Here, $p_W^+$ is the pseudo-inverse of $p$ with respect to the
modified inner product
structure on $C_{d}(X;\R)$ (cf.\ Remarks \ref{rem:surject} and  \ref{rem:MPformula}).
 Similarly, $i_E^+$ is the pseudo-inverse for $i$ defined
using the  modified inner product on $C_{d-1}(X;\R)$.

Therefore, all three pseudo-inverse formulas will follow from one, e.g., from
the formula for the surjection, since the formula for an injection can be obtained by taking transposes and applying duality (Lemma \ref{lem:tree-co-tree}), 
whereas the general pseudo-inverse is obtained by taking the composition. 

 For a tree $T \in {\mathcal F}_d(X)$, the composition 
 \[
 C_{d}(T;\R) @>\subset >> C_{d}(X;\R) @>\partial >> B_{d-1}(X;\R)
 \]
is an isomorphism. Let $\alpha_T$ denote its inverse.
Let $\varphi_{T}: B_{d-1}(X; \R) \rightarrow C_{d}(X; \R)$ be the composition 
 \[
 B_{d-1}(X; \R) @>\alpha_T >\cong > C_{d}(T; \R) @>\subset>> C_{d}(X; \R)\, .
 \]
Then using Theorem \ref{thm:MPformula} one has
 \begin{eqnarray}
\label{network-tree}  
p_W^+ =  \tfrac{1}{\Delta_W}\sum_{T}w_T\varphi_{T}, \qquad 
\Delta_W := \sum_T w_T\, ,
\end{eqnarray}
where the sum is over spanning trees $T \in {\mathcal F}_d(X)$.

Similarly, the pseudo-inverse of the inclusion 
$i: B_{d-1}(X;\R) \rightarrow C_{d-1}(X;\R)$
 is obtained as follows: for a co-tree $L \in {\mathcal F}_{d-1}^{*}(X)$ 
 the composition 
 \[
 B_{d-1}(X;\R) \to  C_{d-1}(X;\R) \to C_{d-1}(X;\R)/C_{d-1}(L;\R)
 \]
 is an isomorphism; let $\beta_L$ denote its inverse. 
 Let $\zeta_{L}: C_{d-1}(X; \R) \rightarrow B_{d-1}(X; \R)$ be
 the composition 
 \[
 C_{d-1}(X; \R) \to C_{d-1}(X;\R)/C_{d-1}(L;\R) @>\beta_L >\cong >  B_{d-1}(X; \R) 
 \]
where the first map is vector space projection.

Then
\begin{eqnarray}
\label{network-cotree} i_{E}^{+} = 
\tfrac{1}{\nabla_E}\sum_{L}\tau_L\zeta_{L}, \qquad
\nabla_E = \sum_L \tau_L\, ,
\end{eqnarray}
where the sum is over spanning co-trees $L \in {\mathcal F}_{d-1}^{*}(X)$.

Combining \eqref{network-tree} and \eqref{network-cotree} with the pseudo-inverse composition property we arrive at a formula which encompasses both the Boltzmann splitting formula  (Theorem \ref{thm:Hodgeprob}) and the higher Kirchhoff projection formula of \cite[thm.~A]{Catanzaro:Kirchhoff}:

\begin{thm}[Kirchhoff-Boltzmann Projection Formula] 
\label{thm:Kirchhoff-Boltzmann} The  pseudo-inverse
of the boundary operator $\partial\: C_d(X;\R) \to
C_{d-1}(X;\R)$ with respect to the  modified inner products
defined by $E$ and $W$ is given by
\begin{equation}
\label{network-tree-cotree} 
\partial_{E, W}^{+} = 
\tfrac{1}{\Delta_W\nabla_E}\sum_{L,T}\tau_L w_T\sigma_{L, T}\, , 
\end{equation}
where $ \sigma_{L, T} = \varphi_{T}\circ\zeta_{L}: C_{d-1}(X; \R) \rightarrow C_{d}(X; \R)$
and the sum is indexed over $L \in {\mathcal F}_{d-1}^{*}(X)$ and  $T \in 
{\mathcal F}_d(X)$.
\end{thm}

\begin{rem}  Theorem \ref{thm:Kirchhoff-Boltzmann} 
gives a concrete expression for the average current 
in the case of periodic stochastic driving (in the adiabatic limit)
\cite{CKS}.
Let $\gamma$ be a smooth $1$-dimensional cycle in the vector space of 
parameters $(E, W)$ and let $\hat x  \in Z_{d-1}(X;\Z)$ be a $(d-1)$-cycle.
Let  $x = [\hat x] \in H_{d-1}(X;\Z)$ be the associated homology class.
 Then the {\it average current} of $(\gamma,[x])$
 is defined to be the $d$-cycle
\[
q := \int_{\gamma}A\,  d\rho^{{\rm B}} \in Z_d(X;\R) \, ,
\]
where $A = A(E,W) = e^{-W}\partial^{*}e^{E}{\mathcal L}_{E, W}^{-1}$ is the operator
of equation \eqref{define-A} and $\rho^{{\rm B}}$ is the higher Boltzmann distribution of $x$ (see \cite{CKS} if $d = \dim X = 1$ and
\cite{Catanzaro:thesis} for $d > 1$.). Then, after 
some straightforward algebraic manipulation using 
Theorem \ref{thm:Kirchhoff-Boltzmann},  we find
\[
q = \sum_{L,T}\sigma_{L, T}(\hat x)\int_{\alpha}d\varrho_{T}^{+} \wedge d\varrho_{L}^{-}\, ,
\]
where $\varrho_{T}^{+} = \varrho_{T}^{+}(W) := w_T/\Delta_T, 
\varrho_{L}^{-} = \varrho_{L}^{-}(E) := \tau_L/\nabla_L$
 and $\alpha$ is any smooth $2$-dimensional chain in the space of parameters 
satisfying $\partial \alpha = \gamma$.
\end{rem}

\end{document}